\documentclass[11pt,a4paper,twoside]{article}
\usepackage{amsthm,amsfonts,amsmath,amscd,amssymb}
\usepackage{latexsym}
\usepackage{euscript}
\usepackage{enumitem}
\usepackage{graphicx}
\usepackage{tikz}
\usepackage{caption}
\usepackage{subcaption}
\usepackage{cite}
\usepackage[utf8]{inputenc}
\usepackage[english]{babel}
\usepackage{xcolor}

\usepackage{jmpag} 

\begin{document}
\title{Busemann Functions in Asymptotically Harmonic Finsler Manifolds}
\author {Hemangi  Shah}
\address{Harish-Chandra Research Institute, Chhatnag Road, Jhunsi, Allahabad 211019, India}
\email{hemangimshah@hri.res.in}
\author {Ebtsam  H. Taha}
\address{Department of Mathematics, Faculty of Science, Cairo University, Giza 12613, Egypt,
\\Harish-Chandra Research Institute, Chhatnag Road, Jhunsi, Allahabad 211019, India}
\email{ebtsam.taha@sci.cu.edu.eg}


\BeginPaper 



\newcommand{\R}{\mathbb{R}}
\newcommand{\Z}{\mathbb{Z}}
\newcommand{\N}{\mathbb{N}}
\newcommand{\Hy}{\mathbb{H}}

\newcommand{\cA}{\mathcal{A}}
\newcommand{\cB}{\mathcal{B}}
\newcommand{\cC}{\mathcal{C}}
\newcommand{\cF}{\mathcal{F}}
\newcommand{\cH}{\mathcal{H}}
\newcommand{\cL}{\mathcal{L}}
\newcommand{\cM}{\mathcal{M}}
\newcommand{\cN}{\mathcal{N}}
\newcommand{\cO}{\mathcal{O}}
\newcommand{\cR}{\mathcal{R}}
\newcommand{\cS}{\mathcal{S}}
\newcommand{\cU}{\mathcal{U}}
\newcommand{\cV}{\mathcal{V}}
\newcommand{\cW}{\mathcal{W}}

\newcommand{\al}{\alpha}
\newcommand{\be}{\beta}
\newcommand{\ga}{\gamma}
\newcommand{\Ga}{\Gamma}
\newcommand{\de}{\delta}
\newcommand{\De}{\Delta}
\newcommand{\ep}{\varepsilon}
\newcommand{\om}{\omega}
\newcommand{\Om}{\Omega}
\newcommand{\si}{\sigma}
\newcommand{\Si}{\Sigma}
\newcommand{\kap}{\kappa}
\newcommand{\la}{\lambda}
\newcommand{\La}{\Lambda}
\renewcommand{\phi}{\varphi}

\renewcommand{\th}{\theta}

\newcommand{\rank}{\operatorname{rank}}
\newcommand{\subcorank}{\operatorname{corank}}
\newcommand{\dist}{\operatorname{dist}}
\newcommand{\diam}{\operatorname{diam}}
\newcommand{\Hd}{\operatorname{Hd}}
\newcommand{\CAT}{\operatorname{CAT}}
\newcommand{\hyp}{\operatorname{H}}
\newcommand{\id}{\operatorname{id}}
\newcommand{\Lip}{\operatorname{Lip}}
\newcommand{\pr}{\operatorname{pr}}
\newcommand{\size}{\operatorname{size}}
\newcommand{\Area}{\operatorname{Area}}
\newcommand{\const}{\operatorname{const}}
\newcommand{\Ric}{\operatorname{Ric}}
\newcommand{\tr}{\operatorname{tr}}
\newcommand{\Id}{\operatorname{Id}}

\newcommand{\hypdim}{\operatorname{hypdim}}
\newcommand{\asdim}{\operatorname{asdim}}
\newcommand{\sh}{\operatorname{sh}}
\newcommand{\s}{\operatorname {sin}}
\renewcommand{\cos}{\operatorname{cos}}
\newcommand{\ba}{\ostatmentperatorname{ba}}
\newcommand{\st}{\operatorname{st}}
\newcommand{\hcone}{\operatorname{C}_h\!}
\newcommand{\cone}{\operatorname{Co}}
\newcommand{\an}{\operatorname{An}}

\newcommand{\crr}{\operatorname{cr}}

\newcommand{\ca}{\operatorname{ca}}

\newcommand{\lev}{\operatorname{lev}}

\newcommand{\es}{\emptyset}
\renewcommand{\d}{\partial}
\newcommand{\di}{\d_{\infty}}
\newcommand{\set}[2]{\{#1:\,\text{#2}\}}
\newcommand{\sm}{\setminus}
\newcommand{\sub}{\subset}
\newcommand{\sups}{\supset}
\newcommand{\un}{\underline}
\newcommand{\ov}{\overline}
\newcommand{\wt}{\widetilde}
\newcommand{\wh}{\widehat}
\newcommand{\hdim}{\dim_h}
\newcommand{\md}{\!\mod}
\newcommand{\anglin}{\angle_{\infty}}
\newcommand{\T}{{\cal T}}
\newcommand{\tm}{\T M}
\newcommand{\rev}[1]{\overleftarrow{#1}}

\newcommand{\dotle}{\dot{\le}}
\newcommand{\dotge}{\dot{\ge}}
\newcommand{\dotsim}{\dot{\sim}}
\def\loc{\mathop{\mathrm{loc}}\nolimits}
\def\div{\mathop{\mathrm{div}}\nolimits}
\def\apeqA{\SavedStyle\sim}
\def\apeq{\setstackgap{L}{\dimexpr.5pt+1.5\LMpt}\ensurestackMath{%
  \ThisStyle{\mathrel{\Centerstack{{\apeqA} {\apeqA}}}}}}


\begin{abstract}In the present paper we investigate Busemann functions in a general Finsler setting as well as in asymptotically harmonic Finsler manifolds. In particular, we show that Busemann functions are smooth on asymptotically harmonic Finsler manifolds.

\key{Busemann function;  asymptote; harmonic Finsler manifold; asymptotically harmonic Finsler manifold.}


\end{abstract}

\section{Introduction}
 Finsler geometry is a generalization of Riemannian geometry which is  richer in content and much wider in scope. Working in the Finsler context may need different techniques that do not exist in the Riemannian framework.  In \cite{harmonicFinsler}, harmonic manifolds have been introduced in the Finsler context. Recently, the study of harmonic  and asymptotically harmonic Finsler manifolds of $(\alpha , \beta)$-type has been discussed in \cite{Taha2}. It is known that Busemann functions play an important role in the investigation of the geometry of noncompact complete Riemannian manifolds with negative sectional curvature and  harmonic manifolds (cf. \cite{Andreev, Shah03}). The convexity of Busemann functions is essential  for the study of Hadamard Riemannian manifolds (cf.  \cite{Hada}).  On the other hand,  Busemann functions have been used in the study of reversible Finsler manifolds of negative flag curvature~\cite{egloff}, the splitting theorems for Finsler manifolds of non-negative Ricci curvature \cite{Osplit} and recently the relation between affine functions and  Busemann functions on complete Finsler manifold are treated in \cite{Bus2021}.  Also, Finsler manifolds whose Busemann functions are convex  have been studied in \cite{Sabau2}. The authors in \cite{Andreev, Kell, Knieper, Shiohama, ShiBanktesh} offered insightful discussions about Busemann functions in both complete Riemannian  and Finslerian manifolds.

Our aim is to analyze Busemann functions in the context of Finsler geometry and then apply the obtained results to study  asymptotic harmonic Finsler (AHF) manifolds.  For example, in a forward complete Finsler manifold we find the relation between Busemann functions associated with two asymptotic rays (in Eqn.~\eqref{two asymptotic rays}). Also, we prove any two rays in an AHF-manifold are asymptotic if and only if the corresponding Busemann functions agree upto a constant (in Theorem \ref{Thm: two asymptotic rays}).  Our results lead to the conclusion that Busemann functions are smooth in an AHF-space (see Theorem \ref{smooth bv}),  which is a generalization of \cite[Theorem 3.1]{ Shah03} from the Riemannian  to the Finsler context. Further, in Proposition \ref{bi-asymptotics}, it has been proved that if the horospheres of an AHF-manifold are minimal, then they have the bi-asymptotic property, that is, their asymptotic geodesics are bi-asymptotic.  

 The structure of the present work is as follows. Section 2 is devoted to some preliminaries needed for better exposition of our work. Thereafter, in \S 3,  we give some properties of the Busemann functions in a connected Finsler manifold without conjugate points. Then, we study the relation  of Busemann functions of asymptotic rays in a forward complete Finsler manifold. Finally, in \S 4, we conclude our work with the exploration of Busemann functions in the case of AHF-manifolds.
 \section{Preliminaries}
We use the following notations:  $M$ denotes an $n$-dimensional, $n>1$,  orientable connected  smooth manifold, 
$(TM, \pi, M)$, or simply $TM$, its tangent bundle and $TM_{0}:=TM\setminus \{0\}$ the tangent bundle with the null section removed. The tangent  space at each $x \in M$ without the zero vector is denoted by  $T_x M_{0}$.  The local coordinates $(x^{i})$ on $M$ induce  local
coordinates $(x^{i},y^{i})$ on $TM$.  Moreover, 
$\partial_i$ and $\dot{\partial}_i$ denote partial differentiation with respect to $x^i$ and $y^i$, respectively. 
 
\begin{definition} \label{Finslerdef}\cite{BCS} A Finsler structure on a manifold $M$  is a mapping $F:TM\rightarrow [0, \infty )$ such that $F$ is $C^\infty$ on $TM_0$,  positively homogeneous of degree one in $y$ and the Hessian matrix $(g_{ij}(x,y))_{1 \leq i,j \leq n}$
is positive definite at each point $y$ of $TM_0$, where $\displaystyle{g_{ij}(x,y):=\frac{1}{2} \dot{\partial}_i\dot{\partial}_j F^2(x,y)}$.
\end{definition}
We refer to \cite{BCS, Shlec} for further reading about Finsler geometry.  A Finsler metric is Riemannian when  $g_{ij}(x,y)$ are functions in $x$ only. Further, a Finsler metric can be characterized in any tangent space $T_{x}M$ by its unit vectors, which form a smooth strictly convex hypersurface $I_x M$ called \emph{indicatrix} at the point $x \in M$.  
 When a Finsler metric is Riemannian, this hypersurface at each point of $M$ is a Euclidean unit sphere. The indicatrix of $F$ is $IM:=\displaystyle\cup_{x\in M}I_{x}M$.

 The distance $d_{F}$ induced by $F$ is defined in $M$ by \cite{Shlec, Tam08}$$d_{F}(p,q):= \inf \left\lbrace\int^{1}_{0} F(\dot{\eta}(t))\, dt\,\ | \, \eta: [0,1]  \rightarrow M,\,  C^{1}\text{ curve joining }p \text{ to } q \right\rbrace.$$
 \begin{remark}
 \begin{enumerate}[label={\upshape(\roman*)},ref={\upshape(\Roman*)},topsep=1pt, leftmargin=5ex, labelwidth=5ex]
 \item The Finsler distance is nonsymmetric, that is, $d_{F}(p,q)\neq d_{F}(q,p)$. In other words, the Finsler distance depends on the direction of the curve. Therefore, the reverse of a general Finsler geodesic  
can not be a geodesic. The non-reversibility property is also reflected in the notion of Cauchy sequence and completeness \cite[\S 6.2]{BCS}.
\item Thus, being different from the Riemannian case,  a positively (or forward) complete Finsler manifold $(M , F)$ is not necessarily negatively (or backward) complete.  The classical Hopf-Rinow theorem splits into forward and backward versions \cite[\S 6.6]{BCS}.  A Finsler metric is   \emph{complete} if it is both forward and backward complete.
\item
Another main difference between Finsler and Riemannian geometries is that in a general Finsler manifold, the exponential map is only $C^{1}$ at the origin of $T_{x} M$ and it is $C^{\infty}$ on $T_{x} M_{0}$.
\end{enumerate}   
\end{remark} 
 
A volume measure $d\mu$ (nondegenerate $n$-volume form) on $M$ can be written in  local coordinates  as  $d\mu = \sigma_{\mu}(x) \,dx^{1} \wedge ... \wedge dx^{n}=\sigma_{\mu}(x) \,dx$, where $\sigma_{\mu}(x)$ is a positive smooth function on $M$. Unlike Riemannian geometry, there are several non-equivalent definitions of volume forms used within Finsler geometry.  The most well known ones are \emph{Busemann-Hausdorff $d\mu_{BH}$ and Holmes-Thompson $d\mu_{HT}$ volume forms} \cite{StrongPrinciple}.  Otherwise stated, we work with arbitrary but fixed volume form $d \mu$. That is, the forthcoming definitions and results hold for either Busemann-
Hausdorff volume form or Holmes-Thompson volume form.

It is known that,  if $F$ is a Finsler structure on $M$, then $F$ induces at each point $x\in M$  a Minkowski norm on $T_x M$. Also, $F^*$, the dual structure of $F$,  induces a Minkowski norm on $T_x^*M$.  That is, $F^*: T^*M \rightarrow \R^+$  is defined, for all $(x,\alpha) \in T^*M$, by $$F^*(x,\alpha) :=\sup\{ \alpha (\xi) \,:\, \xi \in I_xM \}. $$  The dual metric associated to $F^*$ is given by 
  $g^*_{ij}(x,\alpha) :=\frac{1}{2}\,\frac{\partial^2 F^{*2}(x,\alpha)}{\partial \alpha^i \partial \alpha^j}  .$ 
  
   
The \textit{Legendre transformation} $J: TM \rightarrow T^*M$ associated with $F$ is defined, for any point $x \in M$, by 
$J(x,y)= g_{ij}(x,y)\, y^{i}\, dx^{j}, \,\,\forall y \in T_{x}M_{0} \text{ and  }J(0)=0 .$ Let  $J^*: T^*M \rightarrow TM$ defined by
$$J^{*}(x,\alpha)= g^*_{ij}\big( x,\alpha \big)\, \alpha_{i}\, \partial_{j},\,\, \forall \alpha \in T^{*}_{x}M_{0} \text{ and  }J^{*}(0)=0,$$
where $g^*_{ij}(x,\alpha) := g^{ij} (J^*(\alpha)).$

\begin{definition}\cite[\S 3.2]{Shlec} The gradient of a differentiable function $f:M \rightarrow \R$  at a point $x \in M$, where $df(x) \neq 0$,   is defined by 
\begin{equation}\label{graddef}
\nabla f(x)=J^*\big( x,df(x) \big)= g^*_{ij}\big( x,df(x) \big)\,\, \partial_{i} f(x)\, \partial_{j}.
\end{equation}
 $df(x)$ can then be written  as follows
\begin{equation}\label{grad}
df(x,v)= g_{\nabla f(x)}(\nabla f(x), v), \,\, \forall v \in T_{x}M.
\end{equation}
\end{definition}
\begin{remark}
Unlike the Riemannian gradient, the gradient $\nabla f(x)$ is non-linear.  It should be noted that when $df(x) = 0$, the gradient $\nabla f(x)$ is  defined to be zero.
\end{remark}
\begin{definition}\cite{Shlec}
A smooth function $f: M \longrightarrow \mathbb{R}$ is called a Finsler distance if $F(\nabla f)=1$.
\end{definition}
A distance function $r$ defined on an open subset $\Omega$ of $(M,F)$ has some interesting geometric properties. Indeed, $\nabla r$ is a unit vector field on $\Omega$ and it induces a smooth Riemannian metric on $\Omega$ defined by $$\hat{F}(x,v):= \sqrt{g_{\nabla r}(v,v)},\,\, \forall v \in TM.$$ Furthermore, $\hat{F}(\hat{\nabla} r)= F(\nabla r) =1$ by \cite[Lemma 3.2.2]{Shlec}.
\begin{definition}\cite[\S 14.1]{Shlec}
Let $(M,F,d \mu)$ be a Finsler $\mu$-space. 
 For a  $ C^2 $ function $ f$,  the Shen's Laplacian $\Delta f$ of $f$ is defined by $\Delta f=\div_{\mu}(\nabla f)$, that is,
\begin{align}\label{shenlapdef}
\Delta f &=\frac{1}{\sigma_{\mu}(x)} \,\partial_{k} \left[ \sigma_{\mu}(x) \; g^{kl}(x,\nabla f(x)) \; \partial_{l} f \right] \nonumber \\ &=   \left[ g^{kl}(x,\nabla f(x))\, \partial_{k}\left(\log(\sigma_{\mu}(x)\right) +\partial_{k}(g^{kl}(x,\nabla f(x)))\right] \partial_{l} f\\ &+   \, g^{kl}(x,\nabla f(x))\, \partial_{l} \partial_{k} f. \nonumber 
\end{align}
\end{definition}
\begin{remark}
Shen's Laplacian is fully non-linear elliptic differential operator of the second order, cf. \cite{Caponio},  which depends on the measure $\mu$ and it is defined on $U_{f}:=\{x \in M\, |\, df(x)\neq 0\}$ by  $(\ref{shenlapdef})$, and to be zero on $\{x \in M \,| \,df(x)= 0\}$.
\end{remark}
\begin{definition}\cite[\S 14.1]{Shlec}
For $u \in H^1_{loc}(M)$, the weak (or distributional) Laplacian of $u$ is defined  by  
\begin{equation}\label{weaklapdef}
 \int_M \phi\Delta u \,d\mu=-\int_M d\phi(\nabla u) \,d\mu,
 \quad \text{for all}\ \phi \in \cC_c^{\infty}(M).
\end{equation}
\end{definition}
\begin{definition}\cite[\S 14.3]{Shlec} 
The Finsler mean curvature of the level hypersurface $r^{-1}(t)$ at $x \in M$  with respect to $\nabla r_{x}$ is defined by  
\begin{equation}\label{Fmeancurvaturedef}
\Pi_{\nabla r}(x):= \frac{d}{dt}\,\log(\sigma_{x}(t,x^{a}))|_{t=t_{o}}, \text{ for some }  t_{o} \in  \textit{Im}(r).
\end{equation}
\end{definition}
The Finsler Laplancian of a distance function $r$ satisfies $\Delta \, r(x)=\Pi_{\nabla r}(x)$ \cite{curdisvol}.
\begin{definition}\cite{harmonicFinsler}
A forward complete Finsler manifold $(M, F)$ endowed  with a smooth volume measure $d\mu$  is (globally) harmonic  if in polar coordinates the volume density function $\overline{\sigma}_{p}(r,y)$ is a radial function around (each) $ p \in M,$ where $\overline{\sigma}_{p}(r,y):= \frac{{\sigma}_{p}(r,y)}{\sqrt{\det(\dot{g}_p(p,y))}}$  and $\dot{g}_p$ is the restriction of $g$ on the indricatrix $I_{p}M$.  That is,  $\overline{\sigma}_{p}(r,y)$ is independent of $y \in I_{p}M$; thus it can be written  as~$\overline{\sigma}_{p}(r)$.
\end{definition}
\begin{theorem}\label{Charact. Laplancain} \cite{harmonicFinsler}
Let  $(M, F, d\mu)$ be a forward complete Finsler $\mu$-manifold.  The following are equivalent:  $(1)$ $(M, F, d\mu)$ is harmonic,   $(2)$ Shen's Laplancian of a distance function is  radial, $(3)$  the Finsler mean curvature of all geodesic spheres of sufficiently small radii (all radii), expressed in polar coordinates, is a radial function.
\end{theorem}
\begin{definition}\cite{harmonicFinsler} The Finsler mean curvature of horospheres  $\Pi_{\infty}$ is the mean curvature of the Finsler spheres of infinite radius, which is defined by $$\Pi_{\infty} =\displaystyle\lim_{r \to \infty}  \Pi_{\nabla r}(x).$$
\end{definition}
\begin{definition}\cite{harmonicFinsler}
A forward complete, simply connected Finsler $\mu$-manifold $(M, F, d\mu)$ without conjugate points  is called asymptotically harmonic Finsler manifold (or shortly, AHF-manifold) if the Finsler mean curvature of horospheres is a real constant $h$.
\end{definition}
 Thus, a noncompact harmonic Finsler manifold with  constant Finsler mean curvature of horospheres is an AHF-manifold. Examples of AHF-manifold are given in \cite{harmonicFinsler}.  Further, \cite{harmonicbook16,  Knieper, z1} offered perceptive discussions about harmonic and asymptotically Riemannian  manifolds.

\section{Analysis of Busemann Functions with Applications} 
An effective tool to study various topics in differential geometry, such as the structure of harmonic spaces in Riemannian geometry, is  Busemann functions. 
 For more details about Busemann functions   see \cite{z1,Papadopoulos, Pet16, Shah03, ShiBanktesh}  in the Riemannian context and \cite{egloff, Kell,  Osplit, Sabau, ShiBanktesh} in the Finsler context. 
\begin{definition}\cite{Osplit} 
Let $(M,F) $ be a forward complete Finsler manifold. A geodesic $\gamma :[0,\infty] \rightarrow M $ is called a forward ray if  
it is a globally minimizing unit speed Finslerian geodesic, that is, $d_{F}(\gamma (s), \gamma (t))= t-s\, \,\,\, \forall \, s < t$ and $F(\dot{\gamma}) =~1$.
\end{definition}
Now, we recall the definition of {\em Busemann
functions} in the context of Finsler geometry~\cite{Osplit,  Sabau, Shiohama} and discuss some of its general properties.  

Let $(M, F)$ be a forward complete noncompact Finsler manifold without conjugate points, there always exists a forward ray $\gamma :[0,\infty) \rightarrow (M,F) $ emanating  
from each point $p:= \gamma(0) \in M$ \cite{Sabau}.  Associated to the ray $\gamma$, we define the following function 
$$b_{\gamma,t} (x) =  d_{F} (x,\gamma(t)) - t,\,\,\, \forall x \in M,$$ where $d_{F}$ is the Finsler distance  which is nonsymmetric.
\begin{lemma}\label{lema3.2}
 For each $x \in M$, the function $b_{\gamma,t} (x)$ is monotonically decreasing with $t$. Moreover, $b_{\gamma,t} (x)$ is bounded below. 
\end{lemma}
\begin{proof} Let $s,t \in [0,\infty )$ such that  $s < t$, using triangle inequality for the nonsymmetric distance $d$, we have
\begin{align*}
t-s&=d_{F}(\gamma (s), \gamma (t)) \leq d_{F}(\gamma (s), x) + d_{F}(x, \gamma (t)) \\  \Longrightarrow 
\qquad \, \, \,\qquad -s &\leq d_{F}(\gamma (s), x) + d_{F}(x, \gamma (t))-t =d_{F}(\gamma (s), x)+b_{\gamma,t} (x)\\ \Longrightarrow 
 \quad -d_{F}(\gamma (0), x)&\leq b_{\gamma,t} (x)~~~~\text{at}~s=0.
\end{align*}
Hence, $b_{\gamma,t} (x)$ is bounded below by $-d_{F}(\gamma (0), x)$.

To prove the decreasing of  $b_{\gamma,t}$,  let $x \in M$ be arbitrary but fixed and let $s <~t $. Using triangle inequality, we have
\begin{align*}
&d_{F}(x, \gamma (t)) \leq  d_{F}(x,\gamma (s)) + d_{F}(\gamma (s),\gamma (t)) = d_{F}(x,\gamma (s)) + t-s \\
&\Longleftrightarrow  d_{F}(x, \gamma (t))-t \leq d_{F}( x, \gamma (s))-s\Longleftrightarrow
b_{\gamma,t} (x)\leq b_{\gamma,s} (x).
\end{align*}
\vspace{-1.2 cm}
\[\qedhere\]
\end{proof}

In view of Lemma \ref{lema3.2}, the limit of the function $b_{\gamma ,t}(x)$ as $t \rightarrow \infty$ exists. This justifies the following definition.
\begin{definition}
The limit $b_{\gamma}(x)$ of $b_{\gamma ,t}(x)$ is called \textit{ the Busemann function associated to the ray $\gamma$}:
\begin{equation}\label{Busemann function}
b_{\gamma}(x) =  \lim_ {t \rightarrow \infty} b_{\gamma,t} (x)= \lim_ {t \rightarrow \infty} \big( d_{F}(x,\gamma(t)) - t \big).
\end{equation}
\end{definition}
We now give some properties of the Busemann functions:
\begin{proposition}\label{properties of Busemann function}
For a forward complete simply connected Finsler manifold without conjugate points, the following hold:
\begin{enumerate}[label={\upshape(\arabic*)},ref={\upshape\arabic*}]
\item Along the ray $\gamma(t)$, we have  $b_{\gamma} (\gamma(t))=-t,\, \forall t>0$. Therefore, $b_{\gamma}(\gamma(0)) = b_{\gamma}(p) = 0.$
\item $b_{\gamma} $ is $1$-Lipschitz in the sense that  \begin{equation}\label{1-Lipschitz}
-d_{F}(y,x)\leq b_{\gamma}(y)-b_{\gamma}(x)\leq d_{F}(x,y),\,\, x,y \in M.
\end{equation} Hence, $b_{\gamma}$ is differentiable almost everywhere and  $b_{\gamma}$ is uniformly continuous.
\item $ b_{\gamma,t} $ converges to $b_{\gamma}$ uniformly on
each compact subset of $M$. 
\end{enumerate}
\end{proposition} 
\begin{proof} 
$(1)$ follows directly from Eqn.~(\ref{Busemann function}). $(2)$ follows from the triangle inequality. $(3)$ follows from Dini's theorem.
\end{proof}
It should be noted that Proposition \ref{properties of Busemann function} is proved in (cf. \cite{Osplit, Sabau}) under the condition that $(M,F)$ is noncompact forward complete Finsler manifold.  In addition, in the Riemannian case there are similar results to ours in \cite{Knieper, Mare, Sormani} under different assumptions.
\begin{proposition}\cite{curdisvol} One can compute  Busemann functions in a vector space equipped with a Finsler structure $(V,F)$ as follows: for any vector $v \in V,$ the Busemann function $b_v$ associated to the ray $\eta_{v}(t)=tv,\, 0<t< \infty ,$ is given by \begin{equation}\label{calculatingBusemannfn}b_v(y)= - y^{i}\, \frac{\partial F(v)}{\partial y^{i}}. \end{equation}\end{proposition} 
\begin{lemma}\label{properties of Fdistance}
Let $(M,F)$ be a forward complete
Finsler manifold and let $f$ be a Finsler distance on $M$, then the following assertions hold:
\begin{enumerate}[label={\upshape(\arabic*)},ref={\upshape\arabic*}]
\item The level sets of $f$ have no critical points and
viz. $f^{-1}(c)$ for any $c$ are smooth hypersurfaces in $M$.
\item The integral curves of $\nabla f$ are unit speed Finslerian geodesics.
\item The level sets of $f$ are parallel hypersurfaces along the direction $\nabla f$. Consequently, $f$ is linear along the integral curves of $\nabla f$.  
\end{enumerate}
\end{lemma}
\begin{proof}
\begin{enumerate}[label={\upshape(\arabic*)},ref={\upshape\arabic*}]
\item A Finsler distance $f$ on $M$ by its definition means that $F(\nabla f)=1$. Then,  $f$ has no critical  points and the rest follows directly.
\item  This follows from \cite[Lemma 6.2.1]{BCS}. 
\item It follows from \cite[\S 4]{Tam08}. That is, $d_{F}(f^{-1}(t), f^{-1}(s))= s-t ,\, t<s$. Consequently, $f$ is linear along the integral curves of $\nabla f$. Indeed, integrating both sides of $\dot{\eta}(t)= \nabla f \circ \eta$, yields $f({\eta}(t))=t + f(\eta(0)).$
\end{enumerate}\vspace{-0.6 cm}
\[\qedhere\]
\end{proof}
\begin{remark}
In the Finsler context, there is a slight difference in the definition of parallel hypersurfaces. This is due to the nonsymmetry of the distance $d_F$, namely, if a hypersurface $f^{-1}(t)$ is parallel to $f^{-1}(s)$, it does not mean that $f^{-1}(s)$ is parallel to $f^{-1}(t)$, unless the Finsler metric is reversible, cf.\cite{Sabau}.
\end{remark}

\begin{remark} \cite{Papadopoulos} Let  $T$ be a distribution on 
 a compact subset  $\Omega$ of $M$, then, $\forall z \in~\mathbb{Z}^{+}$, define the distributional derivative  
 \[\frac{d^{z}T}{dx^{z}}(\phi)= (-1)^{z} \,T(\phi^{(z)}), \,\, \forall  \phi \in \mathcal{D}(\Omega).\]
  That is, 
\begin{equation}\label{distributional derivative}
\int_{\Omega}  D^{(z)}\,T(\phi) d\mu = (-1)^{z}\int_{\Omega}  T(D^{(z)}\phi) \,d\mu,\, \, \forall \phi \in \mathcal{D}(\Omega).
\end{equation}
\end{remark}
\begin{lemma}\label{Lap.weak.cgt}
$\, \Delta b_{\eta, t} \longrightarrow \Delta b_{\eta}$ as $t \longrightarrow \infty$ in the distributional sense.
\end{lemma}
\begin{proof}
Let $\Omega \subset M$ be a compact subset. 
Since  $b_{\eta , t}$ is a continuous function, then it is locally integrable and therefore it is a distribution on $\Omega$. Let $\phi$ be a test function, then $\Delta \phi $ is a test function as well. Thus,
\begin{equation*}
 \int_{\Omega}  (\Delta b_{\eta ,t} )(\phi)\,d\mu =
\int_{\Omega}  b_{\eta ,t} (\Delta \phi)\,d\mu .
\end{equation*}
Now, taking limit as $t \longrightarrow  \infty$   yields
 \begin{equation*}
\int_{\Omega}  b_{\eta } (\Delta \phi)\,d\mu  = \int_{\Omega} (\Delta  b_{\eta })(\phi)\,d\mu.
 \end{equation*}
 Hence,  $\Delta b_{\eta, t} \longrightarrow \Delta b_{\eta}$ in the distributional sense.
\end{proof}
For the corresponding result in Riemannian geometry, one can check \cite{Knieper, Sormani}. 
 \begin{definition} Let $(M,F)$ be a noncompact forward complete Finsler manifold. 
 Let $\eta: [0, \infty ) \longrightarrow M$ be a ray. Another ray $\zeta : [0, \infty ) \longrightarrow M$ is said to be asymptotic to $\eta$ if there exists a sequence $\{t_{i}\}_{i \in \mathbb{N}} \subset [0,\infty[ $ and a sequence $\{\zeta_{i}\}_{i\in \mathbb{N}}$ such that
\[\lim_{i \rightarrow \infty} \zeta_{i}(t) = \zeta(t),\, \forall\, t\geq0 , \quad \lim_ {i \rightarrow \infty} t_{i} = \infty,\] 
 and 
 \[\zeta_{i} : [0, d_{F}(\zeta(0), \eta(t_{i})] \longrightarrow M \]
is a sequence of minimal geodesics from $\zeta(0)$ to $\eta (t_{i})$ (cf.~\cite{Shah03} for the Riemannian case and \cite{Sabau, Osplit} for the Finsler one). 
 \end{definition}
 The following figure explains the definition:
 \begin{center}
\includegraphics[scale=0.53]{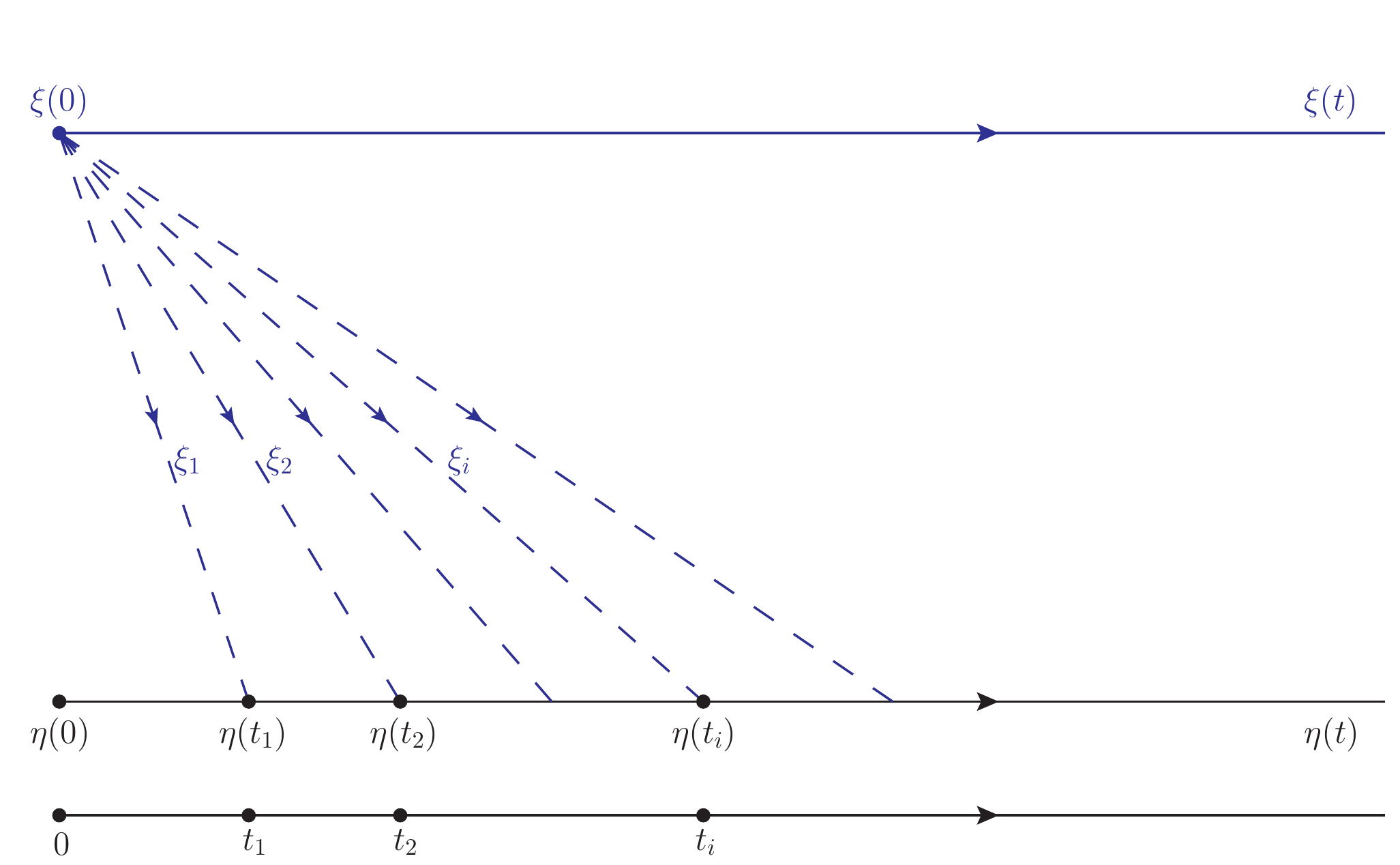}
\end{center}
The following result is a generalization of \cite[Proposition 7.3.8]{Pet16} from the  Riemannian to the Finsler context. 
 \begin{proposition}\label{asymptotic rays}
Let $(M,F)$ be a forward complete Finsler manifold. If a ray ${\zeta}$ emanating from $p:=\zeta(0)$ is asymptotic to $\eta$, then their Busemann functions are related by
\begin{equation}\label{two asymptotic rays}
 b_{\eta}({\zeta}(t)) = b_{\eta}(p) +
{b}_{\zeta}({\zeta}(t)) = b_{\eta}(p) - t.
\end{equation}
Consequently,
\begin{equation}\label{equivalence rel}
 b_{\eta}(x)-b_{\zeta}(x) \leq b_{\eta}(p).
\end{equation}
\end{proposition}
\begin{proof}
Let $\zeta$  be an asymptote to $\eta$ from $p$. Then, there exists a sequence $\{t_{i}\}_{i \in \mathbb{N}} \subset [0,\infty[$ and a sequence $\{\zeta_{i}\}_{i\in \mathbb{N}}$ of minimal geodesics from $p$ to $\eta(t_i)$ such that \[\lim_{i \rightarrow \infty} \zeta_{i}(t) = \zeta(t),\, \forall\, t\geq0  \text{ and } \lim_ {i \rightarrow \infty} t_{i} = \infty.\]
\begin{align*}
 b_\eta(p)& := \lim_{i \rightarrow \infty}
 \big(d_{F}(p,\eta(t_i))-t_i \big)\\
 & = \lim_{ i \rightarrow \infty}
\big( d_{F}(p,\zeta_i(s))+d_{F}(\zeta_i(s),\eta(t_i))-t_i \big)\\
& =  d_{F}(p,{\zeta}(s))+\lim_{i\rightarrow \infty}
\big(d_{F}({\zeta}(s),\eta(t_i))-t_i \big)=  s + b_{\eta}(\zeta(s)).
\end{align*}
That is, $ b_{\eta}(q)-b_{\eta}(p)= - c, \text{ where } \zeta(c)=q,\, c\geq 0$.
Now, we  use Eqn.~(\ref{two asymptotic rays}) to prove (\ref{equivalence rel}) as indicated below. From the triangle inequality for the nonsymmetric distance $d$, we have
\begin{align*}
d_{F}(x,\eta(s)) - s &\leq  d_{F}(x,\zeta(t)) +
d_{F}(\zeta(t),\eta(s)) - s\\ 
&=  d_{F}(x,\zeta(t)) - t +  d_{F}(\zeta(0),\zeta(t)) +
d_{F}(\zeta(t),\eta(s)) -s .
\end{align*}
Now, let $s \longrightarrow  \infty$ in the above inequality, we get
$$b_\eta(x) \leq  d_{F}(x,\zeta(t)) - t + d_{F}(p,\zeta(t)) +
b_\eta(\zeta(t)).$$
Using (\ref{two asymptotic rays}), we get  
 $$b_\eta(x) \leq  d_{F}(x,\zeta(t)) - t + d_{F}(p,\zeta(t)) + b_\eta(p) - t.$$
Therefore,
$$b_\eta(x) \leq  d_{F}(x,\zeta(t)) - t + b_\eta(p) .$$
Taking the limit $t \longrightarrow  \infty$ of both sides, yields (\ref{equivalence rel}).
\end{proof}
Actually, Equation (\ref{two asymptotic rays}) represents a generalization of \cite[Corollary 3.9]{egloff}.

 Now,  we  show that the asymptotes are unique. Consequently, the Busemann functions are  distance functions, i.e., $F(\nabla b_{\zeta})=1$. 
\begin{corollary}\label{asymptote is unique} 
Let $(M,F)$ be a forward complete Finsler manifold. Let  $\eta :[0,\infty] \longrightarrow M $ be a ray and $p \in M$. Then there exists a unique ray $\zeta(s): = \exp_{p}(sv)$ emanating from $p$ that is asymptotic to $\eta$, where $v$ is the initial velocity of  $\zeta$ (cf. \cite{Osplit,  Sabau}).
\end{corollary}
\begin{proof}
Let $\zeta$  be an asymptote to $\eta$ from $p$. 
From Proposition \ref{properties of Busemann function} (2), the Busemann function is differentiable almost everywhere,  and hence one can differentiate both sides of Eqn.~ (\ref{two asymptotic rays}) and get
$$\frac{d}{ds} ( b_\eta({\zeta}(s)) =-1 .$$
Using (\ref{grad}), we have 
\[\frac{d}{ds} ( b_\eta({\zeta}(s))|_{s=0^{+}} = g(\nabla b_{\eta}({\zeta}(s)),
\dot{\zeta}(0)).\] Hence,
\begin{equation}\label{gradbv=1}
\nabla b_{\eta}(p)  = - \dot{\zeta}(0)=-v.
\end{equation}
Therefore, there is only one asymptotic geodesic
to $\eta$ emanating from $p$, namely 
$\zeta(s)= \exp_{p}(s \nabla b_{\eta}(p)).$
\end{proof}

\section{Busemann functions in asymptotically harmonic manifolds}
 It should be noted that for a complete simply connected Riemannian manifold of nonpositive sectional curvature, the asymptotic relation between two rays is an equivalence relation \cite{Pet16}. However, imposing conditions on the flag curvature of a Finsler manifold does not suffice to make it an equivalence relation. This is because the asymptotic relation is neither transitive nor symmetric \cite{ShiBanktesh}. Even-though, we prove the following. 
\begin{theorem}\label{Thm: two asymptotic rays}
In an AHF-manifold, two rays are asymptotic if and only if the corresponding Busemann functions agree upto a constant. Moreover, Eqn.~$(\ref{equivalence rel})$ represents an equivalence relation.   
\end{theorem}
\begin{proof} Let $\zeta$ be a ray asymptotic to ${\eta}$ starting from $p$. Then, it is clear, from (\ref{equivalence rel}) that the function $(b_{\eta} -b_{\zeta})$ attains its maximum at $p$. Let $\Omega$ be a bounded open subset of $M$ containing $p$. Now, applying the \textit{\lq \lq strong comparison principle"} \cite[Lemma 5.4]{StrongPrinciple}  for $u=b_{\zeta}(x) + b_{\eta}(p),\,\,v=b_{\eta}(x)$ and $\Lambda =0$, yields that \begin{equation}\label{eq.rel.}
b_{\zeta}(x) - b_{\eta}(x)=c, 
\end{equation}
where $c=  b_{\eta}(p)$ is a constant and $x $ lies in the component of $\Omega$ containing $p$, say  $x \in U$.
This is because $u \geq v \, \text{ in } \Omega$ and  $\Delta (b_{\zeta}(x) + b_{\eta}(p))= \Delta b_{\zeta}(x) =h,\, \Delta b_{\eta}(x)=h $ as $(M,F, d\mu)$ is an AHF-manifold.
Note that the set $$A:=\{z\in U \,|\,b_{\zeta}(z) - b_{\eta}(z)=c \}$$ is a non-void open bounded subset of $\Omega$. Meanwhile, it is clear that $A$ is a closed set. But our base manifold $M$ is connected, therefore $A$ is the whole $M$.
 One can easily show that the relation~(\ref{eq.rel.}), $\eta \approx \zeta \iff b_{\zeta}(x) - b_{\eta}(x)=  c,$ is an equivalence relation. 
\end{proof}

\begin{corollary} The level sets  $b^{-1}_{\gamma}(t):=(b_{\gamma}(t))^{-1}$  of Busemann function are smooth closed noncompact hypersurfaces of $M$, and called limit spheres or horospheres.
\end{corollary}
\begin{proof}Since, $b_{\gamma}$ is a distance function, then by Lemma \ref{properties of Fdistance} (1),   the level sets of $b_{\gamma}$ have no critical points and viz. $b_{\gamma}^{-1}(c)$ for any $c$ are smooth hypersurfaces in $M$. Moreover, when $M$ is a simply connected, level sets $b_{\gamma}^{-1}(c)$ are noncompact hypersurfaces in $M$. 
\end{proof}
In flat Riemannian manifolds\cite[\S 7.3.2]{Pet16}, horospheres are just affine hyperplanes, and in the case of constant negative sectional curvature, using the Poincare model, horospheres are Euclidean spheres internally tangent to the boundary sphere, minus the point of tangency. This may not be the case in Finsler manifolds (cf.  \cite{harmonicFinsler}).  As we mentioned before, the Busemann function $b_{\gamma} $ is a distance function and $1$-Lipschitz.  Consequently, we can define an AHF-manifold in the weak sense as follows.
\begin{definition}
A forward complete simply connected Finsler $\mu$-manifold $(M, F, d\mu)$ without conjugate points  is called an AHF-manifold in the weak sense if 
 the weak Laplacian of every Busemann function is a real constant, that is $\Delta b_{\gamma} = h$, where $\Delta$ is Shen's Laplacian. 
 \end{definition}
 Another equivalent definition is the following.
 \begin{definition}
A complete simply connected Finsler $\mu$-manifold $(M, F, d\mu)$ without conjugate points  is called an AHF-manifold in the weak sense if 
 the weak forward and backward Laplacians of every Busemann function are real constant. That is,  $\rev{\Delta} b_{\overline{\eta}} =h \text{ and } {\Delta} b_{\eta} =h, \text{ where } h \in \mathbb{R}$, $\Delta$ is Shen's Laplacian  and $\rev{\Delta}$  is the Shen's Laplacian associated with  the \emph{reverse (backward) Finsler structure} $\rev{F}$ of $F$ which is defined by $\rev{F}(v):=F(-v)$.
\end{definition}
It is clear that an AHF-manifold is an AHF-manifold in the weak sense.
\begin{remark}The constant $h$ in the Riemannian case is non-negative, however, in Finsler setting may not be.\end{remark}
 \begin{proposition}For an AHF-manifold in the weak sense, the Finsler mean curvature of large geodesic spheres converges to the Finsler mean curvature of horospheres. 
 \end{proposition}
 \begin{proof}
 It follows from Lemma \ref{Lap.weak.cgt}.
 \end{proof}
 \begin{theorem}\label{smooth bv}
For any ray $\eta$ in an  AHF-manifold in the weak sense, the associated Busemann function $b_{\eta}$ is  smooth. 
\end{theorem}
\begin{proof}
Let $ b_{\eta}$ be the Busemann function associated to the ray $\eta$. We showed, in Proposition \ref{properties of Busemann function}, that $b_{\eta}$ is 1-Lipschitz. Now, assume that $\Delta b_{\eta} (x) = h,\, h \in \mathbb{R}$ in the weak sense. That is, in view of $(\ref{weaklapdef})$,
$$h \int_M \phi \,d\mu =-\int_M d\phi(\nabla b_{\eta}) \,d\mu,
 \qquad \text{for all}\ \phi \in \cC_c^{\infty}(M).$$
 In a coordinate neighborhood $\Omega$,  Shen's Laplacian $(\ref{shenlapdef})$ of $ b_{\eta}$ can be written as follows
$$\Delta b_{\eta} (x) = A^{ij}(x, db_{\eta}{(x)})\,\, \partial_{i}\, \partial_{j} b_{\eta}(x) + B^{i}(x, db_{\eta}{(x)}) \,\,\partial_{i} b_{\eta}(x) 
, $$
where $A^{ij}(x, db_{\eta}(x)):= g^{ij}(x,\nabla b_{\eta}(x)),$ $$\,\,\,\,\,\,\quad B^{i}(x, db_{\eta}(x)):=A^{ij}(x, db_{\eta}(x))\, \partial_{j}\left(\log(\sigma_{\mu}(x)\right) +\partial_{j}(A^{ij}(x, db_{\eta}(x))). $$
It is clear that, the coefficients $A^{ij}(x, db_{\eta}(x)),\,B^{i}(x, db_{\eta}(x))$ are smooth functions for each $x \in M$ and $ db_{\eta}(x) \in T_{x}M_{0}$ since $(M, F, d\mu)$ is a $C^{\infty}$ Finsler manifold equipped with a smooth volume form $d \mu$.
Now,  $\Delta b_{\eta}- h =0$ can be written in the form $\textbf{F}(d^{2}b_{\eta})=0$ on a domain $\Omega $. Thanks to \cite[Theorem 41]{Besse2}, we conclude that $b_{\eta}$ is smooth on $\Omega$.
\end{proof}
\begin{remark}
For a straight line  $\zeta:\mathbb{R} \longrightarrow M$ in a complete Finsler manifold, we have the two associated Busemann functions $b_{\zeta}$ for the forward ray  and $b_{\overline{\zeta}}$ for the backward ray $\overline{\zeta}:= \zeta(-t),\, t \geq 0$, 
\[ b_{\overline{\zeta}}(x)= \lim_{t \rightarrow \infty}  {d_{F}(\overline{\zeta}(t),x) - t}.\]
\end{remark}
Let us recall the definition of \emph{bi-asymptote }in Finsler geometry \cite[\S 4]{Osplit}.
\begin{definition}
 We say that a straight line $\zeta:\mathbb{R} \longrightarrow M$ is bi-asymptotic to $\eta$
if $\zeta|_{[0,\infty)}$ is asymptotic to $\eta|_{[0,\infty)}$ and  $\bar{\zeta}(s)=\zeta(-s)$
is asymptotic to $\bar{\eta}$ with respect to $\rev{F}$.
\end{definition}
\begin{lemma}\label{bv+bv-equal0}
For any straight line $\eta: \mathbb{R} \longrightarrow M$ in an AHF-manifold with $h=0$, the associated  Busemann functions satisfy
\begin{equation}
b_{\eta} +b_{\overline{\eta}} = 0.
\end{equation}
\end{lemma}
\begin{proof}
 The triangle inequality  gives $ b_{\eta} +b_{\overline{\eta}}\geq 0 $ which means that $ b_{\eta} \geq - b_{\overline{\eta}}.$ By direct calculations, $b_{\eta}(\eta(s)) =- b_{\overline{\eta}}(\overline{\eta}(s)) .$
 Let $\Omega$ be a bounded open set of $M$. It is easy to see that $b_{\eta},\,- b_{\overline{\eta}}\in H^{1}(\Omega) \cap C(\Omega)$. Applying the \textit{\lq \lq Strong comparison principle"} \cite[Lemma 5.4]{StrongPrinciple} by putting $u:=b_{\eta},\,v:=- b_{\overline{\eta}}$ and $\Lambda :=0$, yields that $b_{\eta} + b_{\overline{\eta}} \leq 0.$ Hence the result follows.
\end{proof}
\begin{proposition}\label{bi-asymptotics}
 The bi-asymptotics are unique in any AHF-manifold whose Finsler mean curvature of all horospheres  vanishes, i.e., $h=0$.
\end{proposition} 
\begin{proof}
It follows from Lemma \ref{bv+bv-equal0}, using the same technique of the proof of Corollary \ref{asymptote is unique}.
\end{proof}
The next result is a special case of Theorem \ref{smooth bv}. However, the proof is different.
 \begin{proposition}\label{smooth bv,h=0}
For any straight line $\eta$ in an  AHF-manifold $(M,F)$, the associated  Busemann functions $b_{\eta} \text{ and } b_{\overline{\eta}}$ are smooth function on $M$, in case  $h=~0$.
\end{proposition}
\begin{proof}
We have $\Delta b_{\eta} = 0,\, \rev{\Delta} b_{\overline{\eta}}=0$. Since a harmonic function is a static solution to the heat equation and $d b_{\eta}$ does not vanish, then $b_{\eta}$ is smooth by \cite[Proposition 4.1]{Osplit} and \cite[Theorem 4.9 and Remark 4.10]{OShf}. Moreover,  $b_{\eta} + b_{\overline{\eta}} = 0$, by Lemma \ref{bv+bv-equal0}, then  $b_{\overline{\eta}} = - b_{\eta} $ is smooth.
\end{proof}
  
 \begin{proposition}\label{f=bv-}
Suppose $f$ is a distance function on an AHF-manifold. Let $\eta$ be the integral curve of $\nabla f$ starting from $p=\eta(0)$ such that $f(p)=0$. If $\Delta f = h = \Delta b_{\eta}$, then $f = b_{\overline\eta}$.
\end{proposition}
\begin{proof}
From the definition of distance function, $F(\nabla f)=1$. It follows that, the integral curve $\eta$ of $\nabla f$ satisfies $f(\eta(t)) = t $ as $f(\eta (0))= f(p)=0$, by Lemma \ref{properties of Fdistance} (3).     
Now, fix some $s>t$ and let $x \in f^{-1}(s)$,
$$ d_{F}(\eta(t),x) \geq d_{F}(f^{-1}(t),f^{-1}(s)) = s-t, $$     
thereby
$ d_{F}(\eta(t),x) +t \geq s = f(x),\, \forall x \in f^{-1}(s)$.
Hence,
\[\lim_{t \rightarrow -\infty}
\left(d_{F}(\eta(t),x) + t \right)\geq
f(x).\]
Consequently, $b_{\overline{\eta}}(x) \geq f(x),\, \forall x \in f^{-1}(s)$ and $b_{\overline{\eta}}(p)=0=f(p)$. Now, applying \cite[Lemma 5.4]{StrongPrinciple}, we get
$b_{\overline{\eta}}(x) = f(x),\, \forall x \in f^{-1}(s)$. That is, $f = b_{\overline\eta}$.                                                    
\end{proof}

Let $(M,F)$ be a complete simply connected Finsler manifold without conjugate points. Assume that at each $x \in M$ there exists a unique line ${\zeta}$ emanating from $x:=\zeta(0)$ with $\dot{\zeta}(0)=v$. Under these conditions, each $v \in IM$ gives rise to a Busemann function $b_\zeta$, where $\zeta$ is the line just defined above. This justifies the following definition of total Busemann function \cite[\S 5]{ Shah03}.

\begin{definition}
Let $A(M)$ be the set of differentiable functions from $M$ to $\mathbb{R}$.
The total Busemann function
$B : IM \longrightarrow A(M)$ given by $(x,v) \longmapsto b_{(x,v)} := b_\zeta$.  
\end{definition}

\begin{proposition}\label{bd}
Let $\{\gamma _{n} \}\, _{n \in \mathbb{N}}$ be a family of unit speed geodesic rays starting from a fixed point $p \in M$ with initial velocities $\{y_{n}  \}_{n \in \mathbb{N}} \subset I_{p}M$. The sequence $\{b_{y_{n}}\}_{n \in \mathbb{N}}$ is uniformally  bounded on each compact set.
\end{proposition}
\begin{proof}
Let $\Omega$ be a compact subset of M. 
Assume that $\{\gamma _{n} \}\, _{n \in \mathbb{N}}$ is a family of unit speed geodesic rays starting from a fixed point $p \in M$ with initial velocities $\{y_{n} \}\, _{n \in \mathbb{N}}$.  Thus, we have a sequence $\{y_{n}\}_{n \in \mathbb{N}} $  of unit vectors in $I_{p}M$
 and corresponding Busemann functions $\{b_{y_{n}}\}\, _{n \in \mathbb{N}}$. By $(\ref{1-Lipschitz})$, since  each Busemann function $b_{y_{n}}$ is  $1$-Lipschitz, we have 
$$-d_{F}(x_{o},p)\leq b_{y_{n}}(x_o)\leq d_{F}(p,x_{o}).$$

Taking the supreum  over $x_{o} \in \Omega$, where $p \notin \Omega$, we get
\[ -d_{F}(\Omega , p) = \sup_{x_{o}\in \Omega } -d_{F}(x_{o},p) \leq \sup_{x_{o}\in \Omega } b_{y_{n}}(x_o) \leq \sup_{x_{o}\in \Omega } d_{F}(p,x_{o}) = d_{F}(p, \Omega).\]
Hence, $\{b_{y_{n}}\}_{n \in \mathbb{N}}$ is uniformally  bounded.
\end{proof}
\begin{remark}
\begin{enumerate}[label={\upshape(\roman*)},ref={\upshape(\Roman*)},topsep=1pt, leftmargin=5ex, labelwidth=5ex]
\item Consider a sequence of unit vectors $\{y_{n}\}_{n \in \mathbb{N}}$ in $I_{p}M$ such
that $y_n \rightarrow y$. Then, $\{b_{y_{n}}\}_{n \in \mathbb{N}}$ is an equicontinuous family of Busemann functions that is pointwise bounded on each compact subset $\Omega$ of $M$ (Proposition \ref{bd}). 
 Consequently, by \emph{Ascoli-Arzela theorem}, $\{b_{y_n}: \Omega \longrightarrow \mathbb{R}\}_{n \in \mathbb{N}}$ has a uniformly convergent subsequence, say $\{b_{y_{n_k}}\}$, converging to some function~$f$.
\item
 The limit function $f$ is differentiable almost everywhere. Therefore, the gradient $\nabla f$, is defined and the weak Laplacian $\Delta f$ is defined.

 Now, applying the definition of distributional derivative $(\ref{distributional derivative})$ for $T= \nabla b_{v_{n_k}}$ and $z=1$, we get
\begin{align*}
\int_{\Omega} (\Delta b_{v_{n_k}})\, \phi \,d\mu &= -\int_{\Omega}  d\phi(\nabla b_{v_{n_k}}) \,d\mu = \int_{\Omega} b_{v_{n_k}} \Delta \, \phi   \,d\mu ,\\
\int_{\Omega} (\Delta f)\, \phi \,d\mu &= -\int_{\Omega}  d\phi(f)\, d\mu = \int_{\Omega} f\, \Delta \phi \,d\mu.
\end{align*}
Hence, in the distributional sense, we get $$\nabla b_{y_{n_k}} \longrightarrow \nabla f,\,\,\,\Delta b_{y_{n_k}}  \longrightarrow \Delta f,\, \text{ as } {n_k}\longrightarrow \infty.$$
\item
 $(M,F, d\mu)$ being an AHF-manifold, i.e., $\Delta b_{y_{n_k}} =h$. Therefore, $\Delta f=h$ in the distributional  sense.
Now, using the same technique of the proof of  Theorem~\ref{smooth bv}, we deduce that $f$ is smooth.
\item Both $f,\, b_{y_{n_k}} \in C^{\infty}(\Omega, \mathbb{R}).$ Indeed,  $\lim_{{n_k}\rightarrow \infty} b_{y_{n_k}} = f.$
\item Let $\eta_v$ be the integral curve of $\nabla f$ 
 starting from $p:=\eta(0)$. Then, using
Proposition \ref{f=bv-}, we conclude that $f = b_{\overline\eta}$ on $\Omega$.\end{enumerate}
\end{remark}
We end our paper by the following table which summarizes the main differences between Riemannian and Finsler geometries which we have deal with in  the present work.
\begin{center}
{Table 1: Main differences between Riemannian and Finsler geometries}
\end{center}
\fontsize{10}{10}{\begin{center}
\begin{tabular}{|c|c|c|}
\hline
{Geometric}&{Riemannian}&
{Finsler}  \\
{objects}&{manifold $(M,\alpha)$}&
{manifold $(M,F)$}  \\ 
\hline
 Metric  &$\alpha_{ij}(x)$& $g_{ij}(x,y)$\\
\hline	  
Induced distance & $d_{\alpha}(p,q)$ is symmetric & $d_{F}(p,q)$ is non-symmetric\\
\hline
Exponential map &$C^\infty$ on $T_{x} M$& $C^\infty$ on $T_{x} M_{0}$,\\
$\exp_{x}$ && $C^1$ at null section\\
\hline
Legendre transformation & linear& non-linear\\ 
\hline
 Gradient of a function& linear &non-linear
 \\
\hline
Volume measure&canonically defined &several non-equivalent \\
& and unique & (e.g. Holmes-Thompson)\\
\hline
Laplacian&
unique   & not unique \\
&(Laplace-Beltrami),&(e.g.,  Shen's Laplacian),\\
&
 linear elliptic operator&  non-linear elliptic operator\\
\hline
Asymptotic relation is&  for a  simply connected & for a forward complete \\
an equivalence relation & complete $(M, \alpha)$ of  & simply connected $(M, F)$ \\ between two rays  &non-positive sectional & of non-positive flag  \\
& curvature \cite{Pet16}(YES)& curvature (NO)\cite{ShiBanktesh}; \\
&& but (YES) in case \\ 
&&of Theorem \ref{Thm: two asymptotic rays}\\
\hline
Parallel hypersurfaces: &means $f^{-1}(s)$ is & does not mean  $f^{-1}(s)$ \\
$f^{-1}(t)$ is &parallel to $f^{-1}(t)$&is parallel to $f^{-1}(t)$;\\
parallel to $f^{-1}(s)$&&unless $F$ is reversible\\
\hline
Busemann function& 
associated to a ray,&
associated to a forward ray,  \\  
&1-Lipschitz in the sense& 1-Lipschitz in the sense of\\ 
&$|b_{\gamma}(q)-b_{\gamma}(p)|\leq d_{\alpha}(p,q)$& Eqn. \eqref{1-Lipschitz}
\\ \hline
The mean curvature & In AH-Riemannian space &In AH-Finsler space  \\
of horospheres  &is always non-negative&is a real constant \\
\hline
\end{tabular}
\end{center}}
\noindent
{\bf Acknowledgements.} We would like to express our deep thanks to the referees for their careful reading of this manuscript and their valuable comments which led to the present version.



\EndPaper


\end{document}